\documentclass[12pt, reqno]{amsart}

\author[P.~Leonetti]{Paolo Leonetti}
\address{Department of Statistics, Universit\`a Bocconi, via Roentgen 1, Milan 20136, Italy} 
\email{leonetti.paolo@gmail.com}
\urladdr{\url{https://sites.google.com/site/leonettipaolo/}}

\author[M.~Caprio]{Michele Caprio}
\address{Department of Statistical Science,
Duke University, 214 Old Chemistry, Durham, NC 27708-0251}
\email{michele.caprio@duke.edu}
\urladdr{\url{https://mc6034.wixsite.com/caprio}} 


\title{Turnpike in infinite dimension}

\usepackage[T1]{fontenc}
\usepackage{amsmath}
\usepackage{amssymb}
\usepackage{amsthm}
\usepackage[left=2.5cm, right=2.5cm, top=3cm]{geometry}
\usepackage{hyperref}
\usepackage{fancyhdr}
\usepackage[inline]{enumitem}
\usepackage{comment}
\usepackage{nicefrac}
\usepackage{bm}
\usepackage{mathrsfs}
\usepackage{graphicx}
\usepackage[utf8]{inputenc}
\usepackage{cancel}
\usepackage{mathtools}

\newcommand{\vertiii}[1]{{\left\vert\kern-0.25ex\left\vert\kern-0.25ex\left\vert #1 
    \right\vert\kern-0.25ex\right\vert\kern-0.25ex\right\vert}}
		
\AtBeginDocument{%
   \def\MR#1{}
}

\newtheorem{thm}{Theorem}[section]
\newtheorem{cor}[thm]{Corollary}
\newtheorem{lem}[thm]{Lemma}
\newtheorem{prop}[thm]{Proposition}

\theoremstyle{definition} 
\newtheorem{defi}[thm]{Definition}
\let\olddefi\defi
\renewcommand{\defi}{\olddefi\normalfont}
\newtheorem{example}[thm]{Example}
\let\oldexample\example
\renewcommand{\example}{\oldexample\normalfont}
\newtheorem{rmk}[thm]{Remark}
\let\oldrmk\rmk
\renewcommand{\rmk}{\oldrmk\normalfont}

\theoremstyle{remark}
\newtheorem{claim}{\textsc{Claim}}
\newtheorem*{claim*}{\textsc{Claim}}

\hypersetup{
    pdftitle={Turnpike in infinite dimension},
    pdfauthor={Michele Caprio and Paolo Leonetti},
    pdfmenubar=false,
    pdffitwindow=true,
    pdfstartview=FitH,
    colorlinks=true,
    linkcolor=blue,
    citecolor=green,
    urlcolor=cyan
}

\providecommand{\MR}[1]{}

\providecommand{\MR}{\relax\ifhmode\unskip\space\fi MR }

\providecommand{\href}[2]{#2}

\begin{document}

\begin{abstract}
\noindent Let $\Phi$ be a correspondence from a normed vector space $X$ into itself, let $u: X\to \mathbf{R}$ be a function, and $\mathcal{I}$ be an ideal on $\mathbf{N}$. 
Also, assume that the restriction of $u$ on the fixed points of $\Phi$ has a unique maximizer $\eta^\star$. 
Then, we consider feasible paths $(x_0,x_1,\ldots)$ with values in $X$ such that $x_{n+1} \in \Phi(x_n)$ for all $n\ge 0$. Under certain additional conditions, we prove the following turnpike result: every feasible path $(x_0,x_1,\ldots)$ which maximizes the smallest $\mathcal{I}$-cluster point of the sequence $(u(x_0),u(x_1),\ldots)$ is necessarily $\mathcal{I}$-convergent to $\eta^\star$. 

We provide examples that, on the one hand, justify the hypotheses of our result and, on the other hand, prove that we are including new cases which were previously not considered in the related literature. 
\end{abstract}

\maketitle
\thispagestyle{empty}

\section{Introduction}

Let $X$ be a normed real vector space, and fix a correspondence $\Phi$ from $X$ into itself and a functional $u: X\to \mathbf{R}$ which may be interpreted as a utility function. Then, a sequence $\bm{x}=(x_0,x_1,\ldots)$ with values in $X$ is said to be feasible if $x_{n+1}\in\Phi(x_n)$ for all $n\ge 0$. Note that this sequence is simply the orbit of the starting point $x_0$ if $\Phi$ is singleton-valued. Fix also an ideal $\mathcal{I}$ on the nonnegative integers $\mathbf{N}$, which will represent the family of "small" sets (see Section \ref{sec:preparation} for definitions). Assuming that $\bm{x}$ belongs to a given constraint set $\mathscr{C}$ of feasible sequences, we say that $\bm{x}$ is $\mathcal{I}$-optimal if it maximizes the smallest $\mathcal{I}$-cluster point of the real sequence $(u(x_0),u(x_1),\ldots)$; here, an $\mathcal{I}$-cluster point is, informally, an accumulation point which is not small with respect to $\mathcal{I}$. 

Our aim is to study the asymptotic stability of $\mathcal{I}$-optimal paths, which is often referred to as \emph{turnpike property}, see e.g. \cite{MR2164615, MR3966463} for a textbook exposition. 
Roughly, this property states that an $\mathcal{I}$-optimal path spends "most" of the time within a small neighborhood of some optimal stationary point, which is an identified fixed point of $\Phi$. 
Here, following the same lines of \cite{MR3166601, MR1821765, 
MR1772096}, the adjective "most" is intended with respect to the ideal $\mathcal{I}$. In particular, $\mathcal{I}$-optimal paths are  potentially not convergent to the optimal stationary point. 
As remarked in \cite{MR2164615}, the turnpike property has the following interpretation: if one is looking for an optimal way to reach $A$ from $B$ by car, then he should enter onto a turnpike, spend most of the time there, and finally leave the turnpike to reach the claimed point. 
There is an extensive literature which studies this phenomenon, see e.g. \cite{
MR4082481,
MR3217211,
MR3902445,
MR4032435,
MR3334969,
MR1692188
}.

Our main result (Theorem \ref{thm:mainthmfixedpoint}) generalizes the main ones obtained in \cite{MR3166601, MR1821765}. 
We discuss later how our assumptions are related to the ones in these articles. In addition, we show with some novel examples that: 
\begin{enumerate}[label={\rm (\roman{*})}]
\item The turnpike property provided in Theorem \ref{thm:mainthmfixedpoint} does not hold without any restriction on the ideal $\mathcal{I}$ (see Example \ref{example:Itralsation});
\item An $\mathcal{I}$-optimal path may not converge, in the classical sense, to the optimal stationary point (see Example \ref{ex:idealconvergentnotconvergent});
\item The turnpike property holds also in infinite dimension (see Example \ref{ex:infinitedimension}). 
\end{enumerate}
Our main result and its consequences follow in Section \ref{sec:mainresult}.

\subsection{Preparation}\label{sec:preparation} 
An ideal $\mathcal{I}\subseteq \mathcal{P}(\mathbf{N})$ is a family closed under finite union and subsets. 
It is also assumed that $\mathcal{I}$ contains the family of finite sets $\mathrm{Fin}$ and it is different from $\mathcal{P}(\mathbf{N})$. 
Let also 
$\mathcal{I}^\star:=\{S\subseteq \mathbf{N}: S^c\in \mathcal{I}\}$ be its dual filter and
$\mathcal{I}^+:=\{S\subseteq \mathbf{N}: S\notin \mathcal{I}\}$ be the collection of $\mathcal{I}$-positive sets. 
We denote by $\mathcal{Z}$ be the ideal of asymptotic density zero sets, i.e.,
$$
\mathcal{Z}=\left\{A\subseteq \mathbf{N}: |\{a \in A: a\le n\}|=o(n) \text{ as }n\to \infty\right\}.
$$
An ideal $\mathcal{I}$ on $\mathbf{N}$ is said to be \emph{translation invariant} if $(A+k) \cap \mathbf{N} \in \mathcal{I}$ for all $A \in \mathcal{I}$ and all (possibly negative) integers $k \in \mathbf{Z}$. Note that the ideal $\mathcal{Z}$ is translation invariant. 
Classes of translation invariant ideals have been widely studied, see e.g. \cite{MR3863054, MR4074561} and the ideals generated by the upper densities considered in \cite{MR4054777}. 
However, there exist ideals which are not translation invariant: for instance, all the maximal ideals (indeed exactly one between the even and the odd integers belongs to a maximal ideal) and much simpler ones as the family of all sets $A\subseteq \mathbf{N}$ containing finitely many even integers. 


Let $\bm{x}=(x_n)$ be a sequence taking values in a topological vector space $S$. 
Then we say that $\bm{x}$ is $\mathcal{I}$\emph{-convergent} to $\eta \in S$, shortened as $\mathcal{I}\text{-}\lim \bm{x}=\eta$, if $\{n \in \mathbf{N}: x_n \in U\} \in \mathcal{I}^\star$ for all open neighborhoods $U$ of $\eta$. 
Moreover, we say that $\eta \in S$ is an $\mathcal{I}$\emph{-cluster point} of $\bm{x}$ if $\{n \in \mathbf{N}: x_n \in U\} \in \mathcal{I}^+$ for all open neighborhoods $U$ of $\eta$. 
The set of $\mathcal{I}$-cluster points of $\bm{x}$ is denoted by $\Gamma_{\bm{x}}(\mathcal{I})$. 
Usually $\mathcal{Z}$-convergence and $\mathcal{Z}$-cluster points are referred to as \emph{statistical convergence} and \emph{statistical cluster points}, respectively, see e.g. \cite{
MR1181163, MR1260176}. 
Note that $\mathrm{Fin}$-convergence coincides with the ordinary convergence and that $\Gamma_{\bm{x}}(\mathrm{Fin})$ is the set of ordinary accumulation points of $\bm{x}$. 
It is worth noting that $\mathcal{I}$-cluster points have
been studied much before under a different name. 
Indeed, as it follows by \cite[Theorem 4.2]{MR3920799} and \cite[Lemma 2.2]{MR4126774}, they correspond to classical “cluster points” of a filter (depending on $\bm{x}$) on the underlying space, cf. \cite[Definition 2, p.69]{MR1726779}. 

Finally, following \cite{MR1416085}, for each real sequence $\bm{x}$ such that $\{n \in \mathbf{N}: |x_n|\ge M\}\in \mathcal{I}$ for some $M \in\mathbf{R}$, we define its $\mathcal{I}$-limit inferior as 
$$
\mathcal{I}\text{-}\liminf \bm{x}:=\inf\{r \in \mathbf{R}: \{n \in \mathbf{N}: x_n>r\}\in \mathcal{I}^+\}.
$$
Simmetrically, we let $\mathcal{I}\text{-}\limsup \bm{x}:=-\mathcal{I}\text{-}\liminf (-\bm{x})$ be the $\mathcal{I}$-limit superior. 
Again, it is easy to see that if $\mathcal{I}=\mathrm{Fin}$ then they coincide with the ordinary limit inferior and limit superior of $\bm{x}$, respectively. It is remarkable that they can be rewritten also as the smallest and the biggest $\mathcal{I}$-cluster point of $\bm{x}$, respectively, cf. Corollary \ref{lem:representationIliminf} below.

Given sets $A,B$, we say that $\alpha: A\rightrightarrows B$ is a correspondence if $\alpha(x)$ is a (possibly empty) subset of $B$ for each $x \in A$. Moreover, we denote the set of its fixed points by
$$
\mathrm{Fix}(\alpha):=\{x \in A: x \in \alpha(x)\}.
$$
We recall that, if $A$ and $B$ are endowed with some topologies, then 
the correspondence 
$\alpha$ is upper hemicontinuous at $x \in A$ if for each open $U\supseteq \alpha(x)$ there exists an open neighborhood $V$ of $x$ such that $z \in V$ implies $\varphi(z)\subseteq U$. 
Moreover, $\alpha$ is lower hemicontinuous at $x \in A$ if for every open $U\subseteq B$ with $\varphi(x)\cap U\neq \emptyset$ there exists an open neighborhood $V$ of $x$ such that $z \in V$ implies $\varphi(z)\cap U\neq \emptyset$. 
Finally, the correspondence $\alpha$ is said to be \emph{continuous} if it is both upper and lower hemicontinuous at each point $x \in A$, see \cite[Definition 17.2]{MR2378491}. 

Given a function $h: A\to B$ and a sequence $\bm{x}=(x_0,x_1,\ldots)$ with values in $A$, we write $h(\bm{x})$ for the sequence $(h(x_0), h(x_1), \ldots)$. 
Lastly, if $B=\mathbf{R}$, 
we say that $x_0 \in A$ is a \emph{maximizer} of $h$ if $h(x) \le h(x_0)$ for all $x \in A$ (and similarly for minimizers). 




\section{Main result}\label{sec:mainresult}

Let $X$ be a real normed vector space and denote by $\mathcal{K}$ the collection of its nonempty compact subsets. 
Also, let $\mathcal{I}$ be an ideal on the nonnegative integers $\mathbf{N}$,  
and fix a correspondence 
$
\Phi: X\rightrightarrows X
$ 
and a function
$ 
u: X\to \mathbf{R}. 
$ 
In this setting, the function $u$ will take the role of a utility function which induces a total preorder on $X$. 

Let $\mathscr{K}$ be the family of sequences $\bm{x}=(x_0,x_1,\ldots)$ taking values in $X$ which are $\mathcal{I}$-contained in a compact, that is, such that $\{n \in \mathbf{N}: x_n \notin K\}\in \mathcal{I}$ for some  $K\in \mathcal{K}$. 
Moreover, we let $\mathscr{F}$ be the collection of \emph{feasible paths} $\bm{x}$ which satisfy $x_{n+1} \in \Phi(x_n)$ for all $n$, that is, 
$$
\mathscr{F}=\{\bm{x} \in X^{\mathbf{N}}: \forall n \in \mathbf{N}, x_{n+1} \in \Phi(x_n) \}.
$$
It is easy to see that, if $u$ is continuous, then $\mathcal{I}\text{-}\liminf u(\bm{x})$ and $\mathcal{I}\text{-}\limsup u(\bm{x})$ are well defined for each sequence $\bm{x} \in \mathscr{F}_{\mathrm{K}}$, where 
$$
\mathscr{F}_{\mathrm{K}}:=\mathscr{F} \cap \mathscr{K},
$$
cf. Section \ref{sec:preliminaries}. 
Fix also a nonempty subset $\mathscr{C}\subseteq \mathscr{F}_{\mathrm{K}}$, which will take the role of the collection of constraints. 
Note that the primitive elements of this system are represented by the tuple 
$
\langle  X, \Phi, u, \mathcal{I}, \mathscr{C}\rangle.
$ 
Finally, we say that a sequence $\bm{x} \in \mathscr{C}$ is $\bm{\mathcal{I}}$\textbf{-optimal} if  
\begin{equation}\label{eq:definitionoptimal}
\textstyle 
\forall \bm{y} \in \mathscr{C}, \quad 
\mathcal{I}\text{-}\liminf u(\bm{x}) \ge \mathcal{I}\text{-}\liminf u(\bm{y}).
\end{equation}

%
%
In other words, an $\mathcal{I}$-optimal path $\bm{x}$ is a maxmin solution in a precise sense: it maximizes the minimal $\mathcal{I}$-cluster point of the sequence $(u(y_0), u(y_1),\ldots)$ among all feasible paths $\bm{y}$ in the constraint set $\mathscr{C}$, cf. Corollary \ref{lem:representationIliminf} below. 

The aim of this work, in the same spirit of \cite{MR3166601, MR1821765, MR3334969, MR1772096}, is to find sufficient conditions on the system $\langle  X, \Phi, u, \mathcal{I}, \mathscr{C}\rangle$ such that every $\mathcal{I}$-optimal path is necessarily $\mathcal{I}$-convergent to some identified fixed point of $\Phi$ (in this setting, a fixed point of $\Phi$ is usually called \emph{stationary point}).  
We are going to show that certain feasible paths satisfy this property whenever the following conditions on $\langle X, \Phi, u, \mathcal{I}, \mathscr{C}\rangle$ hold: 
\begin{enumerate}[label={\rm (\textsc{A}\arabic{*})}]
\item \label{item:varphicontinuous} $\Phi$ is continuous and takes values in $\mathcal{K}$;

\item \label{item:functioncontinuous} $u$ is continuous; 

\item \label{item:translation} $\mathcal{I}$ is translation invariant;





\item \label{item:restriction} There exists a unique $\eta^\star \in \mathrm{Fix}(\Phi)$ which maximizes the restriction of $u$ on $\mathrm{Fix}(\Phi)$;





\item \label{item:inequality} There exists a continuous linear functional $T: X\to \mathbf{R}$ such that 
\begin{equation}\label{eq:existenceT}
\forall x \in F, \forall y \in \Phi(x), \quad Tx \le Ty \implies x=y=\eta^\star\textup{,}
\end{equation}
where $F:=\{x \in X: u(x) \ge u(\eta^\star)\}$; 

\item \label{item:maximizingsequence} $\sup_{\bm{y} \in \mathscr{C}} \mathcal{I}\text{-}\liminf u(\bm{y}) \ge u(\eta^\star)$.
\end{enumerate} 
We remark that condition \ref{item:varphicontinuous} is equivalent to the fact that the function $X\to \mathcal{K}$ defined by $x\mapsto \Phi(x)$ is continuous with respect to the Hausdorff metric, see \cite[Theorem 17.15]{MR2378491}.

Note that the separation property given in condition \ref{item:inequality} has been already used in \cite{MR3166601, MR1821765} for the case $X=\mathbf{R}^n$, replacing \eqref{eq:existenceT} with the weaker variant
\begin{equation}\label{eq:existenceTweak}
\forall x \in F, \forall y \in \Phi(x), \quad Tx \le Ty \implies x=\eta^\star\textup{.}
\end{equation}
However, a careful analysis of their proofs reveals that, in fact, they were both implicitly using \eqref{eq:existenceT}. 
Indeed, as pointed out by Piotr Szuca in a private communication \cite{MusaSzuca2021}, condition \eqref{eq:existenceTweak} is in fact \emph{not} sufficient for their purposes. 
On this direction, see also Remark \ref{rmk:conditionexistence} below. 
%
%
Condition \eqref{eq:existenceTweak} appeared also in \cite{MR1199717} in the study of turnpike theorems for integral functionals in a continuous time setting. A somehow related condition can be found in \cite[Lemma 4.3]{MR3334969}. 

Our main result follows.
\begin{thm}\label{thm:mainthmfixedpoint}
Let $\langle X, \Phi, u, \mathcal{I}, \mathscr{C}\rangle$ be a system which satisfies conditions \ref{item:varphicontinuous}--\ref{item:maximizingsequence}. 
Let also $\bm{x}\in \mathscr{C}$ be an $\mathcal{I}$-optimal path.  
Then $\mathcal{I}\text{-}\lim \bm{x}=\eta^\star$. 
\end{thm}

It is worth to remark that, differently from most of the literature on turnpike theorems, we do not assume neither the concavity of the utility function $u$ nor the convexity of the images $\Phi(x)$ for each $x \in X$. 
Moreover, a sufficient condition to imply condition \ref{item:maximizingsequence} is the existence of a sequence $\bm{y} \in \mathscr{C}$ which is $\mathcal{I}$-convergent to $\eta^\star$, which gives us the following.
\begin{cor}\label{cor:mainthmfixedpoint1}
Let $\langle X, \Phi, u, \mathcal{I}, \mathscr{C}\rangle$ be a system which satisfies conditions \ref{item:varphicontinuous}--\ref{item:inequality} and suppose that there exists $\bm{y} \in \mathscr{C}$ such that $\mathcal{I}\text{-}\lim \bm{y}=\eta^\star$. 
Let also $\bm{x} \in \mathscr{C}$ be an $\mathcal{I}$-optimal path. 
Then $\mathcal{I}\text{-}\lim \bm{x}=\eta^\star$. 
\end{cor}

Corollary \ref{cor:mainthmfixedpoint1} generalizes the main results obtained in \cite{MR3166601, MR1821765}. Indeed, in \cite{MR1821765} Mamedov and Pehlivan assumed, in addition, that: $X$ is the finite dimensional vector space $\mathbf{R}^k$, $F$ is compact, 
there exists a compact set containing (the image of) each feasible sequence in $\mathscr{F}_{\mathrm{K}}$, the set of contraints $\mathscr{C}$ depends on the sequence $\bm{x}$ so that $\mathscr{C}$ is of the type $\{\bm{z} \in \mathscr{F}_{\mathrm{K}}: x_0=z_0\}$, and there exists $\bm{y}\in \mathscr{C}$ which is convergent to $\eta^\star$ (in the place of the weaker assumption of $\mathcal{I}$-convergence). The same hypotheses have been also used by Das et al. in \cite{MR3166601}, where the authors 
considered certain correspondences $\Phi: \mathbf{R}^k \rightrightarrows \mathbf{R}^k$ such that 
$\Phi(x)=\{h(x,y): y \in U\}$ for all $x \in \mathbf{R}^k$, 
where $h: \mathbf{R}^k \to \mathbf{R}^m$ is a continuous function and $U\subseteq \mathbf{R}^m$ is a fixed nonempty compact set (it is routine to show that all correspondences $\Phi$ of this type are continuous). Lastly, Mamedov and Pehlivan \cite{MR1821765} focused on the case $\mathcal{I}=\mathcal{Z}$. 
Hence, Corollary \ref{cor:mainthmfixedpoint1} proves that all these assumptions are not really needed.

%
%
%

In the next example, we show that Corollary \ref{cor:mainthmfixedpoint1} (and, hence, also Theorem \ref{thm:mainthmfixedpoint}) cannot be extended to all the ideals $\mathcal{I}$. 
\begin{example}\label{example:Itralsation}
Let $\mathcal{I}$ be an ideal on $\mathbf{N}$ such that $2\mathbf{N}\subseteq \mathcal{I}$. Note that such ideals exist, e.g., the family of subsets of $\mathbf{N}$ containing finitely many odd integers, or the maximal ideals extending $2\mathbf{N}$. 
Now, let $X=\mathbf{R}$, and define $\Phi(x)=\{-x,\nicefrac{x}{2}\}$ and $u(x)=x^3$ for each $x \in \mathbf{R}$. 
Moreover, set $\mathscr{C}:=\{\bm{y} \in \mathscr{F}_{\mathrm{K}}: y_0=1\}$. 
Then the continuous correspondence $\Phi$ has a unique fixed point, i.e., $\mathrm{Fix}(\Phi)=\{0\}$. 
It is easily seen that the system $\langle X, \Phi, u, \mathcal{I}, \mathscr{C}\rangle$ satisfies conditions \ref{item:varphicontinuous}--\ref{item:restriction}. 
Moreover, also condition \ref{item:inequality} holds: indeed, notice that $F=\{x \in \mathbf{R}: u(x) \ge u(0)\}=[0,\infty)$. 
Then, setting $T(r)=r$ for all $r \in \mathbf{R}$, we obtain that $Ty<Tx$ for all $(x,y) \neq (0,0)$ with $x \in F$ and $y \in \Phi(x)$, i.e., for all $x>0$ and $y \in \{-x,\nicefrac{x}{2}\}$. 

At this point, let $\bm{x}=(x_0,x_1,\ldots)\in \mathscr{C}$ be the sequence defined by $x_n=(-1)^{n}$ for all $n \in \mathbf{N}$. 
Then, $\bm{x} \in \mathscr{C}$ and $u(\bm{x})=\bm{x}$, so that $\mathcal{I}\text{-}\lim u(\bm{x})=1$. 
Since $|u(y_n)|\le 1$ for all $\bm{y} \in \mathscr{C}$ and $n \in \mathbf{N}$, it follows that $\bm{x}$ is $\mathcal{I}$-optimal. 
Also, the sequence $\bm{y}$ defined by $y_n=2^{-n}$ for all $n \in \mathbf{N}$ belongs to $\mathscr{C}$ and it is convergent (in the classical sense) to $0$. 
Hence, all hypotheses of Corollary \ref{cor:mainthmfixedpoint1} hold. However, the sequence $\bm{x}$ is clearly not $\mathcal{I}$-convergent to $0$.
\end{example}

Note that the same construction given in Example \ref{example:Itralsation} does not contradict Corollary \ref{cor:mainthmfixedpoint1} in the case that $\mathcal{I}$ is a translation invariant ideal. Indeed, in such case, $2\mathbf{N} \in \mathcal{I}$ if and only if $2\mathbf{N}+1\in \mathcal{I}$. However, since their union is $\mathbf{N}$, then they are both $\mathcal{I}$-positive sets, so that $\Gamma_{u(\bm{x})}(\mathcal{I})=\{1,-1\}$ and $\mathcal{I}\text{-}\liminf u(\bm{x})=-1$. Therefore $\bm{x}$ would be not $\mathcal{I}$-optimal. Moreover, this provides an example of a system $\langle X, \Phi, u, \mathcal{I}, \mathscr{C}\rangle$ such that $u$ is neither concave nor convex, $\Phi$ is not convex-valued, and $F$ is not compact.

\medskip

Suppose now that $\Phi$ is singleton-valued, that is, there exists a function $\phi: X\to X$ such that $\Phi(x)=\{\phi(x)\}$ for all $x \in X$. Note that the continuity of $\Phi$ in condition \ref{item:varphicontinuous} is equivalent to the continuity of $\phi$, see \cite[Lemma 17.6]{MR2378491}. Here, let us identify $\Phi$ with $\phi$. Notice also that a feasible sequence $\bm{x} \in \mathscr{F}$ is simply an orbit $(x_0, \phi(x_0), \phi^2(x_0), \ldots)$. Hence the constraint set $\mathscr{C}$ can be identified with the set of starting values $C:=\{x \in X: \exists \bm{x} \in \mathscr{C}, x_0=x\}$.  To sum up, the system can be identified with the tuple $\langle X, \phi, u, \mathcal{I}, C\rangle$, and a sequence $\bm{x}$ with $x_0 \in C$ is $\mathcal{I}$-optimal provided that
$$
\textstyle 
\forall y \in C, \quad 
\mathcal{I}\text{-}\liminf_{n} u(\phi^n(x_0)) 
\ge \mathcal{I}\text{-}\liminf_{n} u(\phi^n(y)),
$$
where $\phi^0(x):=x$ for all $x \in X$. With these premises, we have the following corollary.
\begin{cor}\label{cor:function}
Let $\langle X, \phi, u, \mathcal{I}, C\rangle$ be a system which satisfies conditions \ref{item:varphicontinuous}--\ref{item:inequality} and suppose that there exists $y_0 \in C$ such that $\liminf_n u(\phi^n(y_0))\ge u(\eta^\star)$. 
Fix also $x_0 \in X$ such that the orbit $(\phi^n(x_0))$ is $\mathcal{I}$-optimal. 
Then $\mathcal{I}\text{-}\lim_n \phi^n(x_0)=\eta^\star$. 
\end{cor}

At this point, one may wonder if the standing assumptions of Theorem \ref{thm:mainthmfixedpoint} together with the $\mathcal{I}$-optimality of the sequence $\bm{x}$ imply the stronger conclusion that $\lim \bm{x}=\eta^\star$, so that, in a sense, it would be not necessary to speak about ideals. 
In the next example we show that this is not the case. 
Indeed there exists a system $\langle X, \Phi, u, \mathcal{I}, \mathscr{C}\rangle$ which satisfies conditions \ref{item:varphicontinuous}--\ref{item:maximizingsequence} and an $\mathcal{I}$-optimal sequence $\bm{x} \in \mathscr{C}$ which is not convergent in the ordinary sense (however, thanks to Theorem \ref{thm:mainthmfixedpoint}, it is $\mathcal{I}$-convergent to $\eta^\star$).
\begin{example}\label{ex:idealconvergentnotconvergent}
Set $X=\mathbf{R}$, $\mathcal{I}=\mathcal{Z}$, $\mathscr{C}=\mathscr{F}_{\mathrm{K}}$, $\Phi(x):=[-2x,-\frac{x}{2}]$, and $u(x)=x$ for all $x \in \mathbf{R}$. It is not difficult to see that conditions \ref{item:varphicontinuous}--\ref{item:maximizingsequence} hold and $\mathrm{Fix}(\Phi)=\{0\}$, cf. Example \ref{example:Itralsation}. 

At this point, let $\bm{x}=(x_0,x_1,\ldots)$ be the sequence such that $x_n=(-1)^nz_n$, where $\bm{z}=(z_0,z_1,\ldots)$ is defined as it follows:
$$
\textstyle 
(\overbrace{1,\nicefrac{1}{2}}^{B_1}, \overbrace{1,\nicefrac{1}{2}, \nicefrac{1}{4},\nicefrac{1}{4},\nicefrac{1}{2}}^{B_2}, \ldots, \overbrace{1,\nicefrac{1}{2},\nicefrac{1}{4},\ldots,\nicefrac{1}{2^{k-1}},\underbrace{\nicefrac{1}{2^k},\nicefrac{1}{2^k},\ldots,\nicefrac{1}{2^k}}_{k! \text{ times }},\nicefrac{1}{2^{k-1}},\nicefrac{1}{2^{k-2}},\ldots,\nicefrac{1}{2}}^{B_k},\ldots ).
$$
Here, for each $k\ge 1$, the block $B_k$ has $2k-1+k!$ terms and its middle part is made by $k!$ consecutive terms equal to $\nicefrac{1}{2^k}$. 

Let us show that $\bm{x}$ is $\mathcal{Z}$-optimal: first of all, using the fact that 
$$
\textstyle 
\sum_{k\le n-1}|B_k| \le \sum_{k\le n-1}3\cdot (k-1)!=o(n!) \text{ as }n\to \infty,
$$
thus $\mathcal{Z}\text{-}\lim \bm{x}=0$, cf. also \cite[Lemma 1]{MR4054777}. In particular, $\mathcal{Z}\text{-}\liminf \bm{x}=\mathcal{Z}\text{-}\liminf u(\bm{x})=0$. Let us suppose for the sake of contradiction that there exists a sequence $\bm{y}\in \mathscr{F}_{\mathrm{K}}$ such that $\kappa:=\mathcal{I}\text{-}\liminf \bm{y}>0$. 
Hence $\kappa$ is a statistical cluster point of $\bm{y}$, see Corollary \ref{lem:representationIliminf} below. It follows that 
$$
\textstyle A:=\{n \in \mathbf{N}: y_n>\nicefrac{\kappa}{2}\} \in \mathcal{I}^+.
$$
However, by construction we have that $y_ny_{n+1}<0$ whenever $y_n\neq 0$. 
Therefore $y_{n+1}<-\nicefrac{\kappa}{4}$ for all $n \in A$. 
Considering that $\mathcal{Z}$ is a translation invariant ideal, it follows that $\bm{y}$ is a bounded sequence such that $\{n \in \mathbf{N}: y_n<-\nicefrac{\kappa}{4}\}\supseteq A+1 \in \mathcal{I}^+$. 
We conclude by Lemma \ref{lem:basiccluster}\ref{item:cluster2} below that the sequence $\bm{y}$ has a negative statistical cluster point, contradicting the standing hypothesis that $\mathcal{I}\text{-}\liminf \bm{y}>0$. 

Hence $\bm{x}$ is $\mathcal{Z}$-optimal. However, since the length of block $B_k$ is odd for each $k\ge 2$, it follows that $\liminf \bm{x}=-1$ and $\limsup \bm{x}=1$. 
\end{example}

With these premises, we give below a practical application of our main result in the context of (correspondences generated by) iterated function systems, a basic tool in fractal geometry, see e.g. \cite{MR3236784}. Additional examples can be found also in \cite{MusaSzuca2021}.
\begin{example}\label{ex:fractal}
Set $X=\mathbf{R}$, let $\mathcal{I}$ be a translation invariant ideal on $\mathbf{N}$, and $u: \mathbf{R}\to \mathbf{R}$ be a strictly increasing continuous function. In addition, let $\{\phi_1,\ldots,\phi_k\}$ be an iterated function system on $\mathbf{R}$, that is, a finite number of contractions on $\mathbf{R}$, and define the correspondence $\Phi: \mathbf{R}\rightrightarrows \mathbf{R}$ by
$$
\forall x \in \mathbf{R}, \quad 
\Phi(x):=\{\phi_1(x),\ldots,\phi_k(x)\}.
$$
Accordingly, let $\mathscr{C}$ 
be an arbitrary subset of bounded feasible sequences such that
\begin{equation}\label{eq:Ccondition}
\forall i=1,\ldots,k, \exists x \in \mathbf{R}, \quad 
(x, \phi_i(x), \phi_i^2(x),\ldots) \in \mathscr{C}. 
\end{equation}
(It is remarkable that there exists a unique nonempty compact set $S\subseteq \mathbf{R}$, called \emph{attractor}, such that $\lim_n H^n(\{x_0\})=S$ for all $x_0 \in \mathbf{R}$, in the Hausdorff metric, where $H$ stands for Hutchinson operator defined by $H(A):=
\bigcup_{x \in A}\Phi(x)$ 
for all $A\subseteq \mathbf{R}$, see \cite{MR625600}.)

For each $i=1,\ldots,k$, let $\eta_i$ be the fixed point of $\phi_i$, hence the restriction of $u$ on $\mathrm{Fix}(\Phi)=\{\eta_1,\ldots,\eta_k\}$ is maximized at the unique point $\eta^\star=\max\{\eta_1,\ldots,\eta_k\}$. In particular, conditions \ref{item:varphicontinuous}--\ref{item:restriction} hold. 
At this point, note that $F=
[\eta^\star, \infty)$ and 
\begin{displaymath}
\begin{split}
\forall i=1,\ldots,k, \forall \eta>\eta_i, \quad 
\phi_i(\eta)-\eta_i \le |\phi_i(\eta)-\phi_i(\eta_i)|<
\eta-\eta_i,
\end{split}
\end{displaymath}
hence $\phi_i(\eta)<\eta$ whenever $\eta>\eta_i$. In particular, $\max \Phi(\eta^\star)=\eta^\star$ and $\max \Phi(\eta)<\eta$ for all $\eta>\eta^\star$. Therefore condition \ref{item:inequality} holds letting $T$ be the identity map. Lastly, let $j$ be an index such that $\eta_{j}=\eta^\star$. Then it follows by \eqref{eq:Ccondition} and the Banach contraction theorem that there exists $x \in \mathbf{R}$ such that $(x, \phi_j(x), \phi_j^2(x), \ldots) \in \mathscr{C}$ and $\lim_n \phi^n_j(x)=\eta^\star$. 

We conclude by Corollary \ref{cor:mainthmfixedpoint1} that, if a sequence $\bm{x}$ in the constraint set $\mathscr{C}$ is $\mathcal{I}$-optimal, then $\mathcal{I}\text{-}\lim \bm{x}=\eta^\star$. 
In particular, in the special case $\mathscr{C}=\mathscr{F}_{\mathrm{K}}$, $u(x)=x$, and $\mathcal{I}=\mathrm{Fin}$, we obtain that: if a bounded feasible sequence $\bm{x}$ maximizes its smallest accumulation point, then it is convergent to the maximal fixed point $\eta^\star$ (this could be obtained, of course, also by a direct method.)
%
\end{example}

As a last motivation for the assumptions given in Theorem \ref{thm:mainthmfixedpoint}, we provide in Section \ref{sec:infinite} an example where our main result holds in an infinite dimensional vector space $X$ (we postpone it because of its length). 
The proofs of our results are given in Section \ref{sec:proofs}.



\section{Preliminaries on $\mathcal{I}$-cluster points}\label{sec:preliminaries}

We collect in the next lemma the basic properties of $\mathcal{I}$-cluster points and $\mathcal{I}$-convergence. These properties hold in greater generality, which we do not require here.
\begin{lem}\label{lem:basiccluster}
Let $\bm{x}$ be a sequence taking values in a metric space $S$ and fix an ideal $\mathcal{I}$. Then 
\begin{enumerate}[label={\rm (\roman{*})}]
\item \label{item:cluster1} $\Gamma_{\bm{x}}(\mathcal{I})$ is closed\textup{;}
\item \label{item:cluster2} $\Gamma_{\bm{x}}(\mathcal{I})\cap K\neq \emptyset$, provided that there exists a compact $K\subseteq S$ such that $\{n \in \mathbf{N}: x_n\in K\}\in \mathcal{I}^+$\textup{;}
\item \label{item:cluster3} $\mathcal{I}\text{-}\lim \bm{x}=\eta$ implies $\Gamma_{\bm{x}}(\mathcal{I})=\{\eta\}$\textup{;}
\item \label{item:cluster4} $\mathcal{I}\text{-}\lim \bm{x}=\eta$ if and only if $\Gamma_{\bm{x}}(\mathcal{I})=\{\eta\}$, provided that there exists a compact $K\subseteq S$ such that $\{n \in \mathbf{N}: x_n \in K\} \in \mathcal{I}^\star$\textup{;}
\item \label{item:cluster5} $\Gamma_{\bm{x}}(\mathcal{I})$ is the smallest closed set $C$ such that $\{n \in \mathbf{N}: x_n \in U\}\in \mathcal{I}^\star$ for all open sets $U\supseteq C$, provided that there exists a compact $K\subseteq S$ such that $\{n \in \mathbf{N}: x_n \in K\} \in \mathcal{I}^\star$\textup{.}
\end{enumerate}
\end{lem}
\begin{proof}
See \cite[Lemma 3.1, Corollary 3.2, Corollary 3.4, and Theorem 4.3]{MR3920799}. 
\end{proof}

To the best of authors' knowledge, the following result is the first one of this type, even if some consequences were known, cf. Corollary \ref{lem:representationIliminf} below. 
Informally, it states that, for each sequence contained in a compact, the set of $\mathcal{I}$-cluster points of its continuous image coincides with the continuous image of its $\mathcal{I}$-cluster points. 
\begin{prop}\label{prop:clusterscontinuity}
Let $S,S^\prime$ be metric spaces and let $\mathcal{I}$ be an ideal. Fix also a continuous function $h:S\to S^\prime$ and let $\bm{x}$ be a sequence with values in $S$ such that $\{n \in \mathbf{N}: x_n \in K\} \in \mathcal{I}^\star$ for some compact $K\subseteq S$. 
Then $h(\Gamma_{\bm{x}}(\mathcal{I}))=\Gamma_{h(\bm{x})}(\mathcal{I})$.
\end{prop}
\begin{proof} 
First, suppose that $\eta \in \Gamma_{\bm{x}}(\mathcal{I})$. 
Then it follows by the continuity of $h$ that $h(\eta) \in \Gamma_{h(\bm{x})}(\mathcal{I})$: indeed for each open neighborhood $U$ of $h(\eta)$ there exists an open neighborhood $V$ of $\eta$ such that $\{n \in \mathbf{N}: x_n \in V\} \subseteq \{n \in \mathbf{N}: h(x_n) \in U\}$. 
Hence $h(\Gamma_{\bm{x}}(\mathcal{I}))\subseteq \Gamma_{h(\bm{x})}(\mathcal{I})$.

Conversely, suppose that $\nu \in \Gamma_{h(\bm{x})}(\mathcal{I})$. 
Note that $F:=h^{-1}(\{\nu\})$ is closed and that $\{n \in \mathbf{N}: h(x_n) \in H\} \in \mathcal{I}^\star$, where $H:=h(K)$ is compact. 
Since $\nu$ belongs to $H$ then $K_0:=F \cap K$ is a nonempty compact set. 
We claim that there exists $\eta \in K_0$ which is also an $\mathcal{I}$-cluster point of $\bm{x}$. To show this, for each $r>0$, let $V_r$ be the closed ball with center $\nu$ and radius $r$. Moreover, for each $x \in F$, let $U_{x,r}$ be the open ball with center $x$ and radius $r$. 

Since $h$ is continuous and $K$ is compact, then $G_r:=K\cap h^{-1}(V_r)$ is compact and contains $K_0$ for each $r>0$. 
Let $h_0$ be the restriction of $h$ to the compact set $K$, so that  $h_0$ is uniformly continuous. It follows that $h_0$ admits a modulus of continuity, i.e., there exists a function $\omega: [0,\infty] \to [0,\infty]$ such that 
$$
\textstyle 
\lim_{r\to 0}\omega(r)=0 
\,\,\,\,\,\text{ and }\,\,\,\,\, 
\forall a,b \in K, \quad d^\prime(h_0(a),h_0(b)) \le \omega (d(a,b)),
$$
where $d$ and $d^\prime$ represent the metric on $S$ and $S^\prime$, respectively. 
In particular, $\omega$ is finite in a (right) neighrborhood of $0$, let us say $[0,\varepsilon]$. 
Replacing, if necessary, each $\omega(r)$ with $\sup_{q\le r}\omega(q)$, we can assume without loss of generality that $\omega$ is nondecreasing. Lastly, let $\omega^{-1}$ be the generalized inverse of $\omega$, i.e., $\omega^{-1}(r):=\inf\{q: \omega(q)>r\}$ for each $r>0$. 
For each $r>0$, we obtain that 
$
r\le \sup \omega(d(a,b)),
$  
where the supremum is taken with respect to all $a \in F$ and $b \in U_{a,\omega^{-1}(r)}$. 
This implies that
\begin{equation}\label{eq:basicinclusion}
\textstyle 
\forall r>0, \quad 
G_{r}\subseteq \bigcup_{x \in F}U_{x,\, 2\omega^{-1}(r)}.
\end{equation}
However, since $G_r$ is compact, there exists $x_{r,1}, \ldots, x_{r, m_r} \in F$ such that $G_r$ is contained into $\bigcup_{i\le m_r}U_{x_{r,i}, 2\omega^{-1}(r)}$. 
To conclude, for each $r>0$, we have that $A_r:=\{n \in \mathbf{N}: h(x_n) \in V_r\} \in \mathcal{I}^+$, hence it follows by \eqref{eq:basicinclusion} that 
$$
\textstyle 
A_r\setminus I=\{n \in \mathbf{N}: x_n \in G_r\}\subseteq  \bigcup_{i\le m_r}\{n \in \mathbf{N}: x_n \in U_{x_{r,i}, 2\omega^{-1}(r)}\},
$$
where $I:=\{n \in \mathbf{N}: x_n \notin K\} \in \mathcal{I}$. 
Since $A_r\setminus I \in \mathcal{I}^+$ and $\mathcal{I}$ is closed under finite unions, it follows that for each $t \in \mathbf{N}$ there exists $k(t) \in\{1,\ldots,m_{1/t}\}$ such that 
$$
B_t:=\{n \in \mathbf{N}: x_n \in U_{x_{1/t, k(t)}, 2\omega^{-1}(1/t)}\}\in \mathcal{I}^+.
$$ 
Since $(x_{1/t, k(t)}: t \in \mathbf{N})$ is a sequence in the compact set $K_0$, there exists a convergent subsequence with limit, let us say, $\eta \in K_0$. Considering that $\lim_t 2\omega^{-1}(1/t)=0$ and that for every $r>0$ the set $C_r:=\{n \in \mathbf{N}: x_n \in U_{\eta,r}\}$ contains $B_t$ for every $t$ sufficiently large in the latter subsequence, we obtain that $C_r \in \mathcal{I}^+$ for all $r>0$. In other words, $\eta \in \Gamma_{\bm{x}}(\mathcal{I})$ and $h(\eta)=\nu$. 
This shows that $\Gamma_{h(\bm{x})}(\mathcal{I})\subseteq h(\Gamma_{\bm{x}}(\mathcal{I}))$, concluding the proof.
\end{proof}

\begin{cor}\label{lem:representationIliminf}
Let $S$ be a metric space and $\mathcal{I}$ be an ideal. In addition, fix a continuous function $h: S\to \mathbf{R}$ and a sequence $\bm{x}$ in $S$ such that $\{n \in \mathbf{N}: x_n \in K\} \in \mathcal{I}^\star$ for some compact $K\subseteq S$. Then 
\begin{equation}\label{eq:liminfrepresent}
\textstyle \mathcal{I}\text{-}\liminf h(\bm{x})=\min  \Gamma_{h(\bm{x})}(\mathcal{I})=\min_{\eta \in \Gamma_{\bm{x}}(\mathcal{I})}h(\eta).
\end{equation}
and, simmetrically,
\begin{equation}\label{eq:limsuprepresent}
\textstyle \mathcal{I}\text{-}\limsup h(\bm{x})=\max  \Gamma_{h(\bm{x})}(\mathcal{I})=\max_{\eta \in \Gamma_{\bm{x}}(\mathcal{I})}h(\eta).
\end{equation}
\end{cor}
\begin{proof}
The first equality of \eqref{eq:liminfrepresent} is a consequence of \cite[Corollary 2.3]{MR3955010}, cf. also \cite[Theorem 1$^\prime$]{MR1416085} for the case $\mathcal{I}=\mathcal{Z}$. 
The second equality of \eqref{eq:liminfrepresent} follows directly by Proposition \ref{prop:clusterscontinuity}. The proof for the case $S=\mathbf{R}^n$ and $\mathcal{I}=\mathcal{Z}$ can be found also in \cite[Lemma 3.1]{MR1772096}.

The proof of \eqref{eq:limsuprepresent} is analogous. 
\end{proof}

\section{Proofs}\label{sec:proofs}

\begin{proof}[Proof of Theorem \ref{thm:mainthmfixedpoint}]
Suppose that $\bm{x}$ is an $\mathcal{I}$-optimal path. 
Since $\bm{x} \in \mathscr{C}\subseteq \mathscr{K}$, there exists a compact set $K\subseteq X$ such that $\{n \in \mathbf{N}: x_n \notin K\} \in \mathcal{I}$.  
It follows by \eqref{eq:definitionoptimal}, Corollary \ref{lem:representationIliminf}, and conditions \ref{item:functioncontinuous} and \ref{item:maximizingsequence}, that
$$
\min_{\eta \in \Gamma_{\bm{x}}(\mathcal{I})} u(\eta) 
=\mathcal{I}\text{-}\liminf u(\bm{x})
\ge \sup_{\bm{y} \in \mathscr{C}}\mathcal{I}\text{-}\liminf u(\bm{y})\ge u(\eta^\star),
$$
hence $\Gamma_{\bm{x}}(\mathcal{I})\subseteq K \cap F$, where we recall that $F$ is the closed set $\{x \in X: u(x) \ge u(\eta^\star)\}$. Since $K$ is compact and $\Gamma_{\bm{x}}(\mathcal{I})$ is closed by Lemma \ref{lem:basiccluster}\ref{item:cluster1}, we obtain that $\Gamma_{\bm{x}}(\mathcal{I})$ is compact. In addition, it is nonempty by 
Lemma \ref{lem:basiccluster}\ref{item:cluster2}, therefore $\Gamma_{\bm{x}}(\mathcal{I})\in \mathcal{K}$. 

Suppose that $F=\{\eta^\star\}$. Since $\Gamma_{\bm{x}}(\mathcal{I})$ is a nonempty subset of $F$, it follows that $\Gamma_{\bm{x}}(\mathcal{I})=\{\eta^\star\}$, hence $\mathcal{I}\text{-}\lim \bm{x}=\eta^\star$ by Lemma \ref{lem:basiccluster}\ref{item:cluster4}. 

Let us suppose hereafter that $|F|\ge 2$, so that the linear operator $T$ in \ref{item:inequality} is nonzero. Replacing $T$ with $T/\|T\|$, we can assume without loss of generality that $\|T\|=1$. 
\begin{claim}\label{lem:graphcontinuous}
The map 
$
\textstyle \mathrm{Gr}(\Phi)\to \mathbf{R}: (x,y) \mapsto T(x-y)
$ 
is continuous.
\end{claim}
\begin{proof}
The claimed map can be rewritten as the restriction on $\mathrm{Gr}(\Phi)$ of the composition $T\circ g$, where $g$ is the continuous function $g: X^2\to X: (x,y) \mapsto x-y$. 
\end{proof}

Since $T$ is continuous and $\Phi(x) \in \mathcal{K}$ for each $x \in X$, 
the function 
$$
\widehat{T}: X\to \mathbf{R}: x\mapsto \max_{y \in \Phi(x)}T(y-x)
$$
is well defined. Since the maximum is reached, we have $\widehat{T}x<0$ for all $x \in F\setminus\{\eta^\star\}$ by \ref{item:inequality}.

\begin{claim}\label{claim:Thatcontinuous}
$\widehat{T}$ is continuous and $\widehat{T}\eta^\star=0$. 
\end{claim}
\begin{proof}
Since $\Phi$ is a continuous correspondence by \ref{item:varphicontinuous} and the map defined in Claim \ref{lem:graphcontinuous} is continuous, it follows by Berge's maximum theorem \cite[Theorem 17.31]{MR2378491} that $\widehat{T}$ is continuous. 

For the second part, we obtain by \ref{item:inequality} that $T(y-\eta^\star)< 0$ for all $y \in \Phi(\eta^\star)$ with $y\neq \eta^\star$. Since $\eta^\star \in \mathrm{Fix}(\Phi)$, we conclude that $\widehat{T}\eta^\star=T(\eta^\star-\eta^\star)=0$. 
\end{proof}
%

Since $T$ is continuous and $\bm{x} \in \mathscr{C} \subseteq \mathscr{K}$, it follows that there exists $M\in \mathbf{R}$ such that $\{n \in \mathbf{N}: |Tx_n| \ge M\} \in \mathcal{I}$. Hence, the $\mathcal{I}$-limit inferior and the $\mathcal{I}$-limit superior of the real sequence $(Tx_n)$ are well defined.

\begin{claim}\label{claim:mincluster}
Fix $\eta \in \Gamma_{\bm{x}}(\mathcal{I})$ such that $\mathcal{I}\text{-}\liminf_n Tx_n=T\eta$. Then $\eta=\eta^\star$. 
\end{claim}
\begin{proof}
Assume that there exists $\eta_0 \in \Gamma_{\bm{x}}(\mathcal{I})$ different from $\eta^\star$ such that $\mathcal{I}\text{-}\liminf_n Tx_n=T\eta_0$. 
Since $\Gamma_{\bm{x}}(\mathcal{I})\subseteq F$ and $\eta_0\neq \eta^\star$, then $\widehat{T}\eta_0<0$. Since $\widehat{T}$ is continuous by Claim \ref{claim:Thatcontinuous}, there exist $\varepsilon,\delta>0$ such that $\widehat{T}x<-\varepsilon$ whenever $\|x-\eta_0\|<\delta$. 
Moreover, it can be assumed without loss of generality that $\delta<\nicefrac{\varepsilon}{2}$. At this point, fix $x,y \in X$ such that $\|x-\eta_0\|<\delta$ and $y \in \Phi(x)$, and let $\pi_{y}$ be a minimizer of $\|\pi-y\|$ with $\pi\in \Gamma_{\bm{x}}(\mathcal{I})$. Since $\widehat{T}x<-\varepsilon$, we get
$$
\textstyle
Ty<Tx-\varepsilon
=T\eta_0+T(x-\eta_0)-\varepsilon
\le T\eta_0+\|T\| \|x-\eta_0\|-\varepsilon
<T\eta_0-\nicefrac{\varepsilon}{2}.
$$
At the same time, we have 
$$
\textstyle 
Ty=T\pi_y+T(y-\pi_y)\ge T\eta_0-\|y-\pi_y\|,
$$
which implies that $\|y-\pi_y\|>\nicefrac{\varepsilon}{2}$.

To sum up, if $\|x-\eta_0\|<\delta$ then $\|y-\pi_y\|>\nicefrac{\varepsilon}{2}$ for all $y \in \Phi(x)$. 
Since $\eta_0$ is an $\mathcal{I}$-cluster point of $\bm{x}$, we have $A:=\{n \in \mathbf{N}: \|x_n-\eta_0\|<\delta\} \in \mathcal{I}^+$. Thus, since $\mathcal{I}$ is translation invariant by \ref{item:translation}, then also $A+1 \in \mathcal{I}^+$.  However, considering that $x_{n+1} \in \Phi(x_n)$ for all $n \in A$, we obtain by the preceeding part that $\|x_{n+1}-\pi_{x_{n+1}}\|>\nicefrac{\varepsilon}{2}$. To sum up, the open set 
$
U:=\{z\in X: \exists \eta \in \Gamma_{\bm{x}}(\mathcal{I}), \|z-\eta_0\|<\nicefrac{\varepsilon}{2}\}
$ 
contains $\Gamma_{\bm{x}}(\mathcal{I})$ and it has empty intersection with the $\mathcal{I}$-positive set $A+1$. This contradicts Lemma \ref{lem:basiccluster}\ref{item:cluster5}.
\end{proof}

\begin{claim}\label{claim:maxcluster}
Fix $\eta \in \Gamma_{\bm{x}}(\mathcal{I})$ such that $\mathcal{I}\text{-}\limsup_n Tx_n=T\eta$. Then $T\eta=T\eta^\star$. 
\end{claim}
\begin{proof}
Assume that there exists $\eta_0 \in \Gamma_{\bm{x}}(\mathcal{I})$ such that  $\mathcal{I}\text{-}\limsup_n Tx_n=T\eta$ and $T\eta\neq T\eta^\star$. 
Since $\eta^\star$ is a minimizer of $T\eta$ with $\eta \in \Gamma_{\bm{x}}(\mathcal{I})$ by Claim \ref{claim:mincluster}, then $\kappa:=T(\eta_0-\eta^\star)>0$. 

Since $\widehat{T}$ is continuous and $\widehat{T}\eta^\star=0$ by Claim \ref{claim:Thatcontinuous}, there exist $\varepsilon, \delta \in (0,\nicefrac{\kappa}{4})$ such that $\widehat{T}x<\varepsilon$ whenever $\|x-\eta^\star\|<\delta$. 
Therefore, for each $y \in \Phi(x)$ such that $\|x-\eta^\star\|<\delta$ we obtain
$$
Ty\le Tx+\varepsilon
=T\eta^\star+T(x-\eta^\star)+\varepsilon 
< T\eta^\star+\delta+\varepsilon
<T\eta^\star+\nicefrac{\kappa}{2}
$$
and, at the same time,
$$
Ty=T\eta_0+T(y-\eta_0)
\ge T\eta^\star+\kappa-\|y-\eta_0\|.
$$
Therefore $\|y-\eta_0\|>\nicefrac{\kappa}{2}$ whenever $\|x-\eta^\star\|<\delta$ and $y \in \Phi(x)$. 
It follows that, if $x\in X$ is chosen such that $\|x-\eta\|<\nicefrac{\delta}{2}$ and $\|\eta-\eta^\star\|<\nicefrac{\delta}{2}$ for some $\eta \in \Gamma_{\bm{x}}(\mathcal{I})$, then $\|x-\eta^\star\|\le \|x-\eta\|+\|\eta-\eta^\star\|<\delta$ and hence $\|y-\eta_0\|<\kappa$ for all $y \in \Phi(x)$.

At this point, note that the set $Q:=\{\eta \in \Gamma_{\bm{x}}(\mathcal{I}): \|\eta-\eta^\star\|\ge \nicefrac{\delta}{2}\}$ is compact. Suppose that $Q\neq \emptyset$. By the continuity of $\widehat{T}$, we have $\max_{\eta \in Q}\widehat{T}\eta<0$. 
It follows there exist $\lambda,\tau>0$ such that $\widehat{T}x<-\lambda$ whenever $\|x-\eta\|<\tau$ for some $\eta\in Q$. In addition, it can be assumed without loss of generality that $\tau<\nicefrac{\lambda}{2}$. 
Now, let us suppose that $x\in X$ is chosen such that $\|x-\eta\|<\tau$ for some fixed $\eta\in Q$. Fix $y \in \Phi(x)$; then 
$$
T\eta_0-\|y-\eta_0\|\le Ty <Tx-\lambda
\le T\eta_x+\|x-\eta\|-\lambda
\le T\eta_0-\nicefrac{\lambda}{2},
$$
so that $\|y-\eta_0\|>\nicefrac{\lambda}{2}$. 

Set $\nu:=\min\{\nicefrac{\delta}{2}, \nicefrac{\kappa}{2}, \tau, \nicefrac{\lambda}{2}\}>0$ and fix $x,y \in X$ with $y \in \Phi(x)$. To sum up the previous observations, we have that:
\begin{enumerate}[label={\rm (\roman{*})}]
\item If there exists $\eta \in \Gamma_{\bm{x}}(\mathcal{I})$ such that $\|x-\eta\|<\nu$ and $\|\eta-\eta^\star\|<\nicefrac{\delta}{2}$ then $\|y-\eta_0\|>\nu$;
\item If there exists $\eta \in \Gamma_{\bm{x}}(\mathcal{I})$ such that $\|x-\eta\|<\nu$ and $\|\eta-\eta^\star\|\ge \nicefrac{\delta}{2}$ (so that $Q\neq \emptyset$) then $\|y-\eta_0\|>\nu$.
\end{enumerate}
Putting everything together, 
if $\|y-\eta_0\|\le \nu$ then $\|x-\eta\|>\nu$ for all $\eta \in \Gamma_{\bm{x}}(\mathcal{I})$. 

We conclude as in the proof of Claim \ref{claim:mincluster}: since $A:=\{n \in \mathbf{N}: \|x_{n+1}-\eta_0\|\le \nu\} \in \mathcal{I}^+$ and $\mathcal{I}$ is translation invariant by \ref{item:translation}, then $A-1\in \mathcal{I}^+$. However, $A-1$ is a subset of $\{n \in \mathbf{N}: \forall \eta \in \Gamma_{\bm{x}}(\mathcal{I}), \|x_n-\eta\|>\nu\}$, which belongs to $\mathcal{I}$ thanks to Lemma \ref{lem:basiccluster}\ref{item:cluster5}. This contradiction concludes the proof. 
\end{proof}

To complete the proof, note that by Corollary \ref{lem:representationIliminf}  there exist nonempty compact sets $\Gamma_{\mathrm{min}}, \Gamma_{\mathrm{max}}\subseteq \Gamma_{\bm{x}}(\mathcal{I})$ such that $\mathcal{I}\text{-}\liminf_n Tx_n=T\eta$ for all $\eta \in \Gamma_{\mathrm{min}}$ and $\mathcal{I}\text{-}\limsup_n Tx_n=T\eta$ for all $\eta \in \Gamma_{\mathrm{max}}$. Hence, Claim \ref{claim:mincluster} and Claim \ref{claim:maxcluster} imply that 
$$
\{\eta^\star\}=\Gamma_{\mathrm{min}}
\quad \text{ and }\quad 
\eta^\star \in \Gamma_{\mathrm{min}}\cap \Gamma_{\mathrm{max}}.
$$ 
In other words, the function $\Gamma_{\bm{x}}(\mathcal{I})\to \mathbf{R}$ defined by $\eta \mapsto T\eta$ has a unique point of minimum which is also a maximizer. Therefore $\Gamma_{\bm{x}}(\mathcal{I})=\{\eta^\star\}$, which is equivalent to $\mathcal{I}\text{-}\lim \bm{x}=\eta^\star$ by Lemma \ref{lem:basiccluster}\ref{item:cluster4}.
\end{proof}

\begin{rmk}\label{rmk:conditionexistence}
As it is evident from the proof of Theorem \ref{thm:mainthmfixedpoint}, the full strenght of condition \ref{item:inequality} has not been used. 
Indeed, we needed it only in Claim \ref{claim:mincluster} to show that $\widehat{T}\eta_0<0$ for some $\eta_0 \in \Gamma_{\bm{x}}(\mathcal{I})$ and in Claim \ref{claim:maxcluster} to show that $\max_{\eta \in Q}\widehat{T}\eta<0$ for a suitable subset $Q\subseteq \Gamma_{\bm{x}}(\mathcal{I})$. 
Therefore, it is enough to replace \eqref{eq:existenceT} with the weaker condition 
$$
\forall x \in \Gamma_{\bm{x}}(\mathcal{I}), \forall y \in \Phi(x), \quad Tx \le Ty \implies x=y=\eta^\star.
$$
However, this condition, differently from \ref{item:inequality}, is depends on a given sequence $\bm{x}$.
\end{rmk}

\begin{proof}[Proof of Corollary \ref{cor:mainthmfixedpoint1}] 
Thanks to Theorem \ref{thm:mainthmfixedpoint}, it is sufficient to show that the existence of a sequence $\bm{y} \in \mathscr{C}$ which is $\mathcal{I}$-convergent to $\eta^\star$ implies condition \ref{item:maximizingsequence}. 
To this aim, observe that $\Gamma_{\bm{y}}(\mathcal{I})=\{\eta^\star\}$ by Lemma \ref{lem:basiccluster}\ref{item:cluster3}. 
It follows by Corollary \ref{lem:representationIliminf} that
$$
\mathcal{I}\text{-}\liminf u(\bm{y})
=\min_{\eta \in \Gamma_{\bm{y}}(\mathcal{I})} u(\eta) 
= u(\eta^\star),
$$
concluding the proof.
\end{proof}

\begin{proof}[Proof of Corollary \ref{cor:function}]
It is enough to note that every $\mathcal{I}$-cluster point is an ordinary accumulation point, so that by Corollary \ref{lem:representationIliminf} we obtain
\begin{displaymath}
\begin{split}
\textstyle 
\sup_{y \in C}\mathcal{I}\text{-}\liminf_n u(\phi^n(y))
&\textstyle \ge \mathcal{I}\text{-}\liminf_n u(\phi^n(y_0))\\
&\textstyle =\min \Gamma_{(u(\phi^n(y_0)))}(\mathcal{I}) \ge 
\liminf_n  u(\phi^n(y_0))\ge u(\eta^\star).
\end{split}
\end{displaymath}
Hence condition \ref{item:maximizingsequence} holds, and the conclusion follows by Theorem \ref{thm:mainthmfixedpoint}.
\end{proof}

%
%
%
%
%


\section{An infinite dimensional example}\label{sec:infinite}

As promised, we provide a practical example where Theorem \ref{thm:mainthmfixedpoint} holds in infinite dimension. 

\begin{example}\label{ex:infinitedimension}
Let $X$ be the Hilbert space $\ell_2$ of square summable real sequences, i.e., sequences $\bm{x}=(x_0,x_1,\ldots)$ such that 
$$
\textstyle 
\|\bm{x}\|:=\sqrt{\sum_{i\ge 0}x_i^2}<\infty.
$$ 
Fix a sequence $\bm{x}^\star \in \ell_2$ and define $\mathscr{C}=\{(\bm{x}^{(n)}) \in \mathscr{F}_{\mathrm{K}}: \bm{x}^{(0)}=\bm{x}^\star\}$. 
Let also $\mathcal{I}$ be an arbitrary translation invariant ideal, and, for each $\bm{x} \in \ell_2$, set $u(\bm{x})=x_0$ and 
\begin{equation}\label{eq:definitionPhi}
\textstyle 
\Phi(\bm{x})=\left\{(-\sum_{i\ge 1}x_i^2,y_1,y_2,\ldots ): 2x_i \le y_i \le x_i+\nicefrac{1}{i} \text{ for all }i\ge 1\right\}\cup \{\frac{1}{2}\bm{x}\},
\end{equation}
where $\frac{1}{2}\bm{x}=(\nicefrac{x_0}{2}, \nicefrac{x_1}{2},\ldots)$. First of all, let us show that $\Phi$ is well defined. To do this, fix $\bm{x} \in \ell_2$ and let us prove that $\Phi(\bm{x})\subseteq \ell_2$. If $\bm{y} \in \Phi(\bm{x})$ is equal to $\frac{1}{2}\bm{x}$ then it clearly belongs to $\ell_2$. Otherwise, since $\ell_2$ is a vector space, it is sufficient to show that $\bm{z}=\bm{y}-\bm{x}=(-x_0-\sum_{i\ge 1}x_i^2,z_1,z_2,\ldots ) \in \ell_2$ where $x_i \le z_i \le \nicefrac{1}{i}$ for all $i\ge 1$. Therefore
\begin{displaymath}
\begin{split}
\textstyle 
\sum_{i\ge 0}z_i^2\le z_0^2+\sum_{i\ge 1}\left(|x_i|+\frac{1}{i}\right)^2 
&\textstyle \le z_0^2+\sum_{i\ge 0}x_i^2+\sum_{i\ge 1}\frac{1}{i^2}+2\sum_{i\ge 1}\frac{|x_i|}{i} \\
&\textstyle \le z_0^2+\sum_{i\ge 0}x_i^2+\sum_{i\ge 1}\frac{1}{i^2}+2\sqrt{\sum_{i\ge 1}x_i^2 \cdot \sum_{i\ge 1}\frac{1}{i^2}}<\infty,
\end{split}
\end{displaymath}
where the last $\le$ follows by the Cauchy--Schwarz inequality. 
In addition, $\Phi(\bm{x})$ is compact. 
To this aim, since $\{\frac{1}{2}\bm{x}\}$ is compact, it is sufficient to show that a translation of the first set in the definition \eqref{eq:definitionPhi} of $\Phi(\bm{x})$ is compact. Let us define 
\begin{equation}\label{eq:definitionvarphi}
\textstyle \varphi(\bm{x}):=\left\{(-x_0-\sum_{i\ge 1}x_i^2,z_1,z_2,\ldots ): x_i \le z_i \le \nicefrac{1}{i} \text{ for all }i\ge 1\right\}.
\end{equation}
Let $\bm{a}$ be the sequence defined by ${a}_0:=-x_0-\sum_{i\ge 1}x_i^2$ and ${a}_i:=|x_i|+\frac{1}{i}$. Note that $\bm{a} \in \ell_2$ and that $\varphi(\bm{x})$ is a closed subset of $\{\bm{z} \in \ell_2: |z_i| \le {a}_i \text{ for all }i\ge 0\}$. However, the latter set is compact thanks to \cite[p. 453]{MR1009162}, hence $\varphi(\bm{x})$ is compact too. 
To sum up, $\Phi(\bm{x})$ is compact subset of $\ell_2$ which is nonempty (since it contains $\frac{1}{2}\bm{x}$). 

At this point, let us show that $\Phi$ is continuous. Reasoning as above, it is sufficient to show that the correspondence $\varphi$ defined in \eqref{eq:definitionvarphi} is continuous at $\bm{x}$. Assume that $\varphi(\bm{x})\neq \emptyset$, i.e., $x_i \le \nicefrac{1}{i}$ for all $i\ge 1$, otherwise the claim is trivial. 

First, let us show that $\varphi$ is upper hemicontinuous. Fix $\varepsilon>0$ and define the open set $U_\varepsilon:=\{\bm{z} \in \ell_2: \exists \bm{y} \in \varphi(\bm{x}), \|\bm{z}-\bm{y}\|<\varepsilon\}$. We need to find a constant $\delta>0$ such that, for each $\bm{x}^\prime \in \ell_2$, if $\|\bm{x}-\bm{x}^\prime\|<\delta$ then $\varphi(\bm{x}^\prime) \subseteq U_\varepsilon$.  Hence, fix also $\bm{x}^\prime \in \ell_2$ such that $\|\bm{x}-\bm{x}^\prime\|<\delta$ for a suitable $\delta>0$ that will be chosen later and pick $\bm{y}^\prime \in \varphi(\bm{x}^\prime)$. 
In particular, $|x_i-x_i^\prime|<\delta$ for all $i \ge 0$. 
Similarly, we can assume without loss of generality that $x_i^\prime\le \nicefrac{1}{i}$ for all $i\ge 1$, otherwise $\varphi(\bm{x}^\prime)=\emptyset$. 
Then
\begin{equation}\label{eq:estimatex0}
\begin{split}
\textstyle 
\forall \bm{y} \in \varphi(\bm{x}), \quad |y_0-y_0^\prime|
&\textstyle \le |x_0-x_0^\prime|+|\sum_{i\ge 1}x_i^2-\sum_{i\ge 1}(x_i^\prime)^2\,|\\
&\textstyle \le |x_0-x_0^\prime|+|x_0^2-(x_0^\prime)^2|+|\|\bm{x}\|^2-\|\bm{x}^\prime\|^2\,| \\
&\textstyle \le \delta+\delta |x_0+x_0^\prime|+\delta(\|\bm{x}\|+\|\bm{x}^\prime\|) \\
&\textstyle \le \delta\left(1+2 (\|\bm{x}\|+\|\bm{x}^\prime\|)\right) \\
&\textstyle \le \delta\left(1+4\|\bm{x}\|+2\delta\right).
\end{split}
\end{equation}
Now, recall that for each integer $i\ge 1$ we have $x_i^\prime \le y_i^\prime \le \nicefrac{1}{i}$ and $x_i \le y_i \le \nicefrac{1}{i}$. In particular, there exists $\bm{y} \in \varphi(\bm{x})$ such that $|y_i-y_i^\prime|\le |x_i-x_i^\prime|$ for all $i\ge 1$. It follows that 
\begin{equation}\label{eq:estimatex1}
\textstyle 
\sum_{i\ge 1}(y_i-y_i^\prime)^2\le \sum_{i\ge 1}(x_i-x_i^\prime)^2 \le \|\bm{x}-\bm{x}^\prime\|^2 \le \delta^2.
\end{equation}
Putting together the above estimates we obtain that, for each given $\bm{y}^\prime \in \varphi(\bm{x}^\prime)$, there exists $\bm{y} \in \varphi(\bm{x})$ such that 
\begin{equation}\label{eq:finalestimate}
\textstyle 
\|\bm{y}-\bm{y}^\prime\|=\sqrt{|y_0-y_0^\prime|^2+\sum_{i\ge 1}(y_i-y_i^\prime)^2}\le \delta \sqrt{\left(1+4\|\bm{x}\|+2\delta\right)^2+1}<\varepsilon,
\end{equation}
where the last inequality holds if $\delta$ is sufficiently small. 

Second, let us show that $\varphi$ is lower hemicontinuous. To this aim, fix an arbitrary $\bm{y} \in \varphi(\bm{x})$ and some $\varepsilon>0$. We claim that there exists $\delta>0$ such that if $\|\bm{x}-\bm{x}^\prime\|<\delta$ then $\|\bm{y}-\bm{y}^\prime\|<\varepsilon$ for some $\bm{y}^\prime \in \varphi(\bm{x}^\prime)$. Note that estimates \eqref{eq:estimatex0} and \eqref{eq:estimatex1} hold simmetrically also in this case, with the conclusion that we have exactly the same upper bound computed in \eqref{eq:finalestimate} for $\|\bm{y}-\bm{y}^\prime\|$. 

%
Therefore $\Phi$ is a nonempty compact-valued continuous correspondence on $\ell_2$, $\mathcal{I}$ is translation invariant, and $u$ is clearly continuous. Hence conditions \ref{item:varphicontinuous}--\ref{item:translation} hold. 

The set of fixed points of $\Phi$ is neither convex nor compact and it is equal to
$$
\textstyle 
\mathrm{Fix}(\Phi)=\{\bm{x} \in \ell_2: x_0=-\sum_{i\ge 1}x_i^2 \text{ and }x_j\le 0 \text{ for all }j\ge 1\}
$$
It follows that the restriction of $u$ on $\mathrm{Fix}(\Phi)$ has a unique maximizer, which is the zero sequence $\bm{0}$ of $\ell_2$, hence condition \ref{item:restriction} holds. 
Moreover, we have that 
$$
F=\{\bm{x} \in \mathscr{C}: u(\bm{x}) \ge u(\bm{0})\}=\{\bm{x} \in \mathscr{C}: x_0 \ge 0\}.
$$
Fix sequences $\bm{x} \in F$ and $\bm{y} \in \Phi(\bm{x})$ such that $(\bm{x}, \bm{y}) \neq (\bm{0}, \bm{0})$. Note that $\Phi(\bm{0})=\{\bm{0}\}$, hence $\bm{x}\neq \bm{0}$ (in particular, if $x_0=0$, then $x_i\neq 0$ for some $i\ge 1$). Thus, setting $T=u$ (which is a continuous linear functional on $\ell_2$), we obtain that $T\bm{y}=-\sum_{i\ge 1}x_i^2<x_0=T\bm{x}$, i.e., condition \ref{item:inequality} holds. 
Lastly, note that the sequence $(\bm{x}^\star, \frac{1}{2}\bm{x}^\star, \frac{1}{4}\bm{x}^\star, \ldots)$ is convergent to $\bm{0}$ and belongs to $\mathscr{C}$ (indeed, it is starts at $\bm{x}^\star$, it is feasible, and its image is contained in the compact set $\{k\bm{x}^\star: k \in [0,1]\}$). 

We conclude by Corollary \ref{cor:mainthmfixedpoint1} that each $\mathcal{I}$-optimal sequence $(\bm{x}^{(n)}) \in \mathscr{C}$ in the system $\langle\, \ell_2, \Phi, u, \mathcal{I}, \mathscr{C}\rangle$ is necessarily $\mathcal{I}$-convergent to $\bm{0}$. 
\end{example}


\subsection{Acknowledgements.} 
The authors are grateful to Marek Balcerzak (Lodz University of Technology, PL) and Piotr Szuca (University of Gda\'{n}sk, PL) for useful discussions and for letting us know about the manuscript \cite{MusaSzuca2021}. P.L. is grateful to PRIN 2017 (grant 2017CY2NCA) for financial support.

\bibliographystyle{amsplain}
\bibliography{turnpike}

\end{document}